\documentclass[12pt,leqno]{article}  
\usepackage[a-1b]{pdfx}
\usepackage{amssymb,amsmath,amsthm,calc,array,mathrsfs, mathtools, rotating, cite, upgreek, url}
\usepackage{arydshln}
\usepackage{tikz}
\usetikzlibrary{cd,shapes,arrows.meta}
\usepackage[latin1, utf8]{inputenc}
\usepackage{accents}
\usepackage{makeidx}
\usepackage{xstring}
\usepackage{pifont}
\hypersetup{hidelinks}
\usepackage[nameinlink, capitalise]{cleveref}
\usepackage{fancyhdr, titlesec, fancyvrb}
\usepackage[nottoc]{tocbibind}
\usepackage{thmtools}
\usepackage{dashbox}
\usepackage{array}
\usepackage{ stmaryrd }
\usepackage[paper=a4paper,left=25mm,right=25mm,top=25mm,bottom=25mm]{geometry}
\usepackage[all]{xy}
\newcommand\blfootnote[1]{%
  \begingroup
  \renewcommand\thefootnote{}\footnote{#1}%
  \addtocounter{footnote}{-1}%
  \endgroup
}
\newcommand{\bq}{\begin{quote}\begin{footnotesize}}
\newcommand{\eq}{\end{footnotesize}\end{quote}} 
\newcommand{\SSS}{\mathrm S} 
\newcommand{\Q}{\mathbb{Q}}
\newcommand{\C}{\mathbb{C}}
\newcommand{\F}{\mathbb F}
\newcommand{\Fu}[1]{\F_{\! #1}}
\newcommand{\nrm}{\trianglelefteqslant}
\newcommand{\nrme}{\lhd}
\newcommand{\ohne}{\smallsetminus}
\newcommand{\al}{\alpha}
\newcommand{\lm}{\lambda}
\newcommand{\w}{\tilde}
\newcommand{\ti}{\times}
\DeclareMathOperator{\car}{char}
\renewcommand{\=}{\;=\;}
\newcommand{\<}{\backslash}
\newcommand{\hsp}[1]{\hspace*{#1mm}}
\newcommand{\vsp}[1]{\vspace*{#1mm}}
\newtheorem{Satz}{Satz} 
\newtheorem*{theorem}{Theorem}
\newtheorem{Definition}[Satz]{Definition}
\newtheorem{Lemma}[Satz]{Lemma}
\newtheorem{Remark}[Satz]{Remark}
\newtheorem{Example}[Satz]{Example}
\newtheorem{Proposition}[Satz]{Proposition}
\newtheorem{Corollary}[Satz]{Corollary}
\newtheorem{Theorem}[Satz]{Theorem}
\DeclareMathOperator{\Aut}{Aut}
\DeclareMathOperator{\sgn}{sgn}
\newcommand{\Z}{\mathbb{Z}}
\newcommand{\NN}{\mathrm N}
\newcommand{\dell}{\partial}
\newcommand{\auf}{\stackrel}
\newcommand{\mb}[1]{{\mbox{#1}}}
\newcommand{\scr}{\scriptsize}
\newcommand{\scm}{\scriptstyle}
\newcommand{\aufgl}[1]{\auf{\mb{\scr #1}}{=}}
\newcommand{\sollgl}{\aufgl{!}}
\newcommand{\Disj}{\bigsqcup}
\newcommand{\CC}{\mathrm C}
\newcommand{\DD}{\mathrm D}
\newcommand{\la}{\langle}
\newcommand{\rl}{\rangle} 
\renewcommand{\le}{\leqslant}
\renewcommand{\ge}{\geqslant}
\renewcommand{\phi}    {\varphi}
\newcommand{\spi}[1]{\la #1\rl}
\newcommand{\xraiso}[1]{\xrightarrow[\raisebox{1.5mm}{$\smash{\scriptstyle\sim}$}]{#1}}
\newcommand{\ru}[1]{\rule[#1mm]{0mm}{0mm}}
\newcommand{\signum}{\operatorname{sgn}}
\DeclareMathOperator{\id}{id}
\newcommand{\enger}{\setlength{\arraycolsep}{1pt}
                           \renewcommand{\arraystretch}{0.5} }
\newcommand{\weiter}{\setlength{\arraycolsep}{3pt}
                            \renewcommand{\arraystretch}{1.0} }
\newcommand{\smatev}[4]{\enger
                           \left(
                           \begin{array}{cccc}
                            \scm #1  & \scm #2  & \scm #3  & \scm #4  \\
                           \end{array}
                           \right) 
                           \weiter }
\newcommand{\smated}[3]{\enger
                           \left(
                           \begin{array}{ccc}
                            \scm #1 & \scm #2 & \scm #3 \\
                           \end{array}
                           \right) 
                           \weiter }
\newcommand{\smatez}[2]{\enger
                           \left(
                           \begin{array}{cc}
                            \scm #1 & \scm #2 
                           \end{array}
                           \right) 
                           \weiter }
\begin{document}
\title{Absolutely irreducible modules via James triples}
\author{Nora Krau{\ss}}
\date{}
\maketitle
\bq
The construction of absolutely irreducible $KG$-modules for a field $K$ and a finite group $G$ is a basic problem in representation theory. For the symmetric group $\SSS_n$, Gordon James \cite{Jam} gave an explicit solution using subgroup data associated to a $\mu$-tableau~$t$ for partitions $\mu\vdash n$. Here, we want to investigate in which way James's method can be extended beyond symmetric groups. We identify suitable subgroup data from a finite group and use it to construct a module such that an associated factor module is absolutely irreducible or zero. 
\eq

\section{Introduction}
Absolutely irreducible modules play a central role in the representation theory of finite groups. Given a field $K$ and a finite group $G$, a $KG$-module is absolutely irreducible if it is irreducible and remains irreducible after extending the base field to an algebraic closure. 

One of the main computational approaches to obtain irreducible modules is provided by the MeatAxe algorithms \cite{Parker1984}\cite{Parker1998}\cite{HoltRees1994}, which are may be used to decompose e.g.\ the regular $KG$-module $KG$ into its composition factors. A more elaborate algorithm due to Steel using splitting and extension techniques is implemented in Magma \cite{Steel2012}.

An approach using data of subgroups is given by Shoda pairs \cite{Shoda1933}. They give a group theoretic framework for constructing primitive idempotents of group algebras. For $K=\Q$ and $G$ being a nilpotent group, a complete set of orthogonal primitive idempotents can be constructed explicitly using strong Shoda pairs, as shown in the work of Jespers, Olteanu, and del Río \cite{Jesp}. In particular, these idempotents generate all irreducible $\Q G$-modules in this case.

For the symmetric group $\SSS_n$, James \cite{Jam} gave an explicit solution using subgroup data associated to a $\mu$-tableau~$t$ for a partition $\mu\vdash n$. From the corresponding row and column stabilisers, he constructed the Specht modules and obtained a complete set of absolutely irreducible $\Q\SSS_n$-modules and, via factor modules, a complete set of irreducible $\Fu{p}\SSS_n$-modules where $p$ is a prime. 

Inspired by James's submodule theorem for the symmetric group, we introduce James triples $(U,V,N)$ of subgroups of a finite group $G$. To a James triple we associate a subfactor of a permutation module and show that it is either zero or absolutely irreducible. The concept of James triples bears similarities to Shoda pairs, but is not directly related.

A {\it James triple} in $G$ of index $k$ is a triple $\mathcal{T}=(U,V,N)$ of subgroups in $G$ such that the conditions (J\,1,\,2,\,3) hold.
\begin{itemize}
\item[(J\,1)] We have $U\cap V\le N$.
\item[(J\,2)] We have $N\nrme V$ such that $|V/N|=k$. 
\item[(J\,3)] For $h\in G\ohne (UV)$ and $w\in V$, we have 
\[
|wN\cap U^h| \= |N\cap U^h|\; ~.
\]
\end{itemize}

Consider the permutation $KG$-module $M_K = K(U\< G)$ and let $\alpha_K: K\xrightarrow{\sim} K$ be a field automorphism. Then $M_K$ carries a canonical non-degenerate $G$-invariant sesquilinear form with respect to $\alpha_K$. Moreover, let $f$ be a group morphism from $V$ to the group of roots of unity $\upmu(K)$ that is compatible with the structure of $K$ and $(U,V,N)$ in a certain sense.

We then let
$	S^{\mathcal{T},\, f}_K \; :=\;\left\langle\sum\limits_{v\in V}Uv\cdot vf\right\rangle_{\hsp{-1}KG}\;\le\; M_K\,.$

Using the canonical non-degenerate $G$-invariant sesquilinear form on $M_K$ to form $\left(S_K^{\mathcal{T},\, f}\right)^\bot$, we obtain the following.
\begin{theorem}\,
	The $KG$-module $$S^{\mathcal{T},\, f}_K/(S^{\mathcal{T},\, f}_K\cap {\left(S^{\mathcal{T},\, f}_K\right)}^{\bot})$$ is absolutely irreducible or zero.
\end{theorem}

If $K$ is a subfield of $\mathbb{C}$ that is closed under complex conjugation $\alpha_{\mathbb{C}}$, then \mbox{$S^{\mathcal{T},\, f}_K\cap {\left(S_K^{\mathcal{T},\, f}\right)}^\bot=0~$} whence $S^{\mathcal{T},\, f}_K$ is absolutely irreducible.

This article is based on Chapter 1 of the author's thesis \cite{NK}. 

\section{Construction of irreducible modules via James triples}
Let $G$ be a finite group. Let $K$ be a field. Let $\alpha_K\in\Aut(K)$ be an automorphism of $K$. We write $\bar\lambda:=\lambda\alpha_K$ for $\lambda\in K$.
We write $\upmu(K):=\{x\in K: \exists s\in\Z_{\ge 1} \text{ such that } x^s=1\}$ for the group of roots of unity of $K$.

\subsection{The Scalar Lemma}\label{setting}
Let $\mathcal{T}=(U,V,N)$ be a triple of subgroups in $G$ such that the following conditions hold.	
	\begin{itemize}
		\item[(J\,1)] We have $U\cap V\le N$.
		\item[(J\,2)] We have $N\nrme V$. 
	\end{itemize}	
	Conditions (J\,1,\,2) can be summarized in the following diagram. 
	
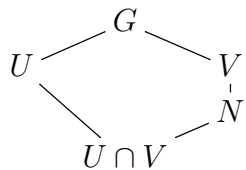
\[
	\xymatrix@C-1.0em@R-1.8em{
		& G                                           &                      \\
		U\ar@{-}[ur] &                                             & V \ar@{-}[ul]        \\
		&                                             & N\ar@{-}[u] \\
		& U\cap V\ar@{-}[uul]\ar@{-}[ur] &                      \\
	}
	\]	The index $k:=\frac{|V|}{|N|}\in\Z_{\ge 2}$ is called the {\it index} of $\mathcal{T}=(U,V,N)$.
	
	Let \[\begin{array}{rcl}
	T_{(U,V,N)}&:=&\{h\in G\ohne(UV): \text{we have }|wN\cap U^h| \= |N\cap U^h| \text{ for $w\in V$}\}\subseteq G\ohne(UV)\\
	T'_{(U,V,N)}&:=& T_{(U,V,N)}\cup (UV)\subseteq G~.\end{array}\]	
	We often abbreviate $T:=T_{(U,V,N)}$ and $T':=T'_{(U,V,N)}$.

	\begin{Definition}
		\label{DefGI1}\label{bil}\rm
	Let $f : V\to\upmu(K)$ be a group morphism.	The morphism $f$ is said to {\it fit} the triple $(U,V,N)$ with respect to $\alpha_K: K\xrightarrow{\sim} K$ if the following conditions (1, 2) hold.
	\begin{enumerate}
	\item[(1)]We have $\overline{vf}=(vf)\al_K=(vf)^{-1}=v^{-1}f$ for $v\in V$.
	\item[(2)]\begin{enumerate}
	\item[(i)] If $\car(K)$ does not divide $|V/N|$, then $N\le\ker(f)< V~.$
	\item[(ii)]If $\car(K)$ does divide $|V/N|$, then $N\le\ker(f)\le V~.$
	\end{enumerate}
	\end{enumerate}
	\end{Definition}

\begin{Lemma}\label{muK}\label{fitL}\label{RemGI5_5}\indent
Suppose given a group morphism $f:V\to\upmu(K)$ fitting  $\mathcal{T}=(U,V,N)$ with respect to $\al_K$. Then {\rm(1, 2, 3)} hold.
\begin{enumerate}
\item[\rm(1)] Let $L|K$ be a field extension. Suppose given an automorphism $\alpha_L\in\Aut(L)$ such that $\alpha_L$ restricts to \mbox{$\al_L|_K^K=\al_K$.} The morphism $\w f: V\to\upmu(L):\, v\mapsto vf$ fits the triple $(U,V,N)$ with respect to $\alpha_L$.
\item[\rm(2)] We have $\sum_{w\in V/N} wf=0~.$
\item[\rm(3)] Suppose given $g\in UV$. Suppose given $u,\,\w u\,\in\, U$ and $v,\,\w v\,\in\, V$ such that \mbox{$g = uv = \w u \w v$.} Then we have  $vf = \w v f$.
\end{enumerate}
\end{Lemma}
{\it Proof.} 
{\it Ad} (1). We have $(v\w f)\al_L=(vf)\al_L=(vf)\al_K=(vf)^{-1}=v^{-1}f=v^{-1}\w f$ for $v\in V$, so condition (1) of Definition \ref{DefGI1} holds.

We have $\car(L)=\car(K)$, so, as $\ker(f)=\ker(\w f)$, condition (2) of Definition \ref{DefGI1} holds.
 
{\it Ad} (2). Regardless of $\car(K)$, if $N\le\ker(f)< V$ then $|(V/N)f|>1$ and hence $\sum\limits_{x\in V/N}x f=0$.
 
 If $\car(K)$ does divide $|V/N|$ and $\ker(f)=V$ then $ \sum\limits_{x\in V/N}xf=\sum\limits_{x\in V/N} 1= |V/N|=0$.
 
{\it Ad} (3).  We have $\w u^{-1}u =  \w v v^{-1} \in U\cap V\;.$ As $U\cap V\le N\le\ker(f)$, we have $(\w v v^{-1})f=1$ and thus $vf = \w v f$.
 \qed

\begin{Definition}\label{ses} We consider the permutation $KG$-module  $M_K := K(U\< G)$ belonging to $U$. Let $m_1 := U\cdot 1 \in M_K\,$. So $m_1\cdot g = Ug\in M_K$ for $g\in G$. We equip $M_K$ with the following non-degenerate $G$-invariant sesquilinear form with respect to $\al_K$:
\[\hspace*{-5mm}
\begin{array}{rcrclcl}
{[-,=]_K=[-,=]} & : & M_K          & \ti & M_K         & \to     & K \\
&   & \left(m_1\cdot g \right. & ,   & \left. m_1\cdot h\right) & \mapsto & [m_1\cdot g,m_1\cdot h] \;:=\;\dell_{m_1\cdot g,\,m_1\cdot h}~,
\end{array} 
\] where $g,h\in G$.
Explicitly, for $x=\sum\limits_{g\in G} g\lambda_g$, $y=\sum\limits_{h\in G} h\nu_h\in KG$, we have
\[\begin{array}{rclcl}
[m_1\cdot x,m_1\cdot y] &=& [m_1\sum\limits_{g\in G} g\lambda_g, m_1\sum\limits_{h\in G} h\nu_h]
						&=& \sum\limits_{g,h\in G} [m_1\cdot g, m_1\cdot h]\lambda_g\overline{\nu_h} \\
						&=& \sum\limits_{g,h\in G}\dell_{m_1\cdot g,\,m_1\cdot h} \lambda_g\overline{\nu_h}
						&=& \sum\limits_{\begin{subarray}{c} g,h\in G\\Ug\,=\,Uh\end{subarray}} \lambda_g\overline{\nu_h}~.
\end{array}\]
\end{Definition}

\begin{Remark}\label{possesMK}\rm
Let $\al_\C:\C\to\C$ be the complex conjugation. Suppose that $K|\mathbb{Q}$ is a subfield of $\mathbb{C}$ that is closed under the complex conjugation $\alpha_{\mathbb{C}}$, i.e.\ $\kappa\alpha_{\mathbb{C}}\in K$ for $\kappa\in K$. Suppose that $\alpha_K=\alpha_{\mathbb{C}}|_K^{K}$\,.
Note that the $G$-invariant sesquilinear form $[-,=]$ on $M_K$ with respect to $\al_K$ is not only non-degenerate, but positive definite.
\end{Remark}

\begin{Definition}\rm
For $h\in G$, we define the $K$-linear coefficient function 
\[
\begin{array}{rcrcl}
\chi_{Uh} & : & M_K          & \to     & K                         \\
&   & m  & \mapsto & {[m,m_1\cdot h]}~. 
\end{array}
\] In particular, $(m_1\cdot g)\chi_{Uh}=\dell_{Ug,Uh}$ for $g\in G$.
\end{Definition}

\begin{Lemma}\label{LemGI5}
Suppose given $z\in G\ohne (UV)$. Suppose given $h\in UzV$ with $h\in T$.

For $w,\,\w w\,\in\, V$, we have
\[
|wN\cap h^{-1}Uz| \= |\w w N\cap h^{-1}Uz|\; .
\]
\end{Lemma}
{\it Proof.} Write $z = uhv$ for some $u\in U$ and some $v\in V$. 
Note that \mbox{$h^{-1}Uz = h^{-1} U uhv = U^h\, v $.}
So the assertion is equivalent to having $|wN\,\cap\, U^h\, v | = |\w w N\,\cap\, U^h\, v|$ for $w,\,\w w\,\in\, V$, i.e.\ to having $|wNv^{-1}\cap U^h| = |\w wNv^{-1}\cap U^h|$ for $w,\w w\,\in\, V$. 

As $N\nrme V$, this is equivalent to having $|wv^{-1}N\cap U^h| = |\w wv^{-1}N\cap U^h|$ for $w,\,\w w\in V$, i.e.\ to having $|wN\cap U^h| = |\w wN\cap U^h|$ for $w,\w w\,\in\, V$, which holds by assumption on $h$.\qed

\begin{Lemma}[Scalar Lemma]
\label{LemGI7}
Suppose given a group morphism $f: V\to\upmu(K)$ fitting $\mathcal{T}=(U,V,N)$. For $h\in T'_{(U,V,N)}$, we set using Lemma \ref{RemGI5_5}.(3)
\[
\lm_h \; :=\; 
\left\{
\begin{array}{ll}
v^{-1} f & \text{if $h\in UV$, where $h = uv$ with $u\in U$ and $v\in V$} \\
0           & \text{if $h\in T=T'\ohne (UV)$.}                                      \\        
\end{array}
\right.
\] Then
\[
\sum\limits_{v\in V} U\!hv\cdot vf \= \lm_h\cdot \sum\limits_{v\in V} Uv\cdot vf\; 
\] holds for $h\in T'_{(U,V,N)}$.

\end{Lemma}

{\it Proof.} We need to show that for $h\in T'$ we have 
\[
\left(\sum\limits_{v_0\in V} U\!hv_0\cdot v_0f\right)\chi_{Uz} \;\sollgl\; \left(\lm_h\cdot \sum\limits_{v_0\in V} Uv_0\cdot v_0f\right)\chi_{Uz}\; 
\]
for $z\in G$.

To evaluate the equation to be shown, we distinguish four cases.\vsp{1}

{\it Case 1: $h\in UV$ and $z\in UV$.}

Write $h=uv$ and $z=\w u\w v$ with $u,\w u\in U$ and $v,\w v\in V$.

On the right-hand side, we obtain 
\[\left(\lm_h \sum\limits_{v_0\in V} Uv_0\cdot v_0f\right)\chi_{Uz}
=\lm_h\sum\limits_{\begin{subarray}{c}v_0\in V\\ Uv_0=U\w v\end{subarray}}v_0 f
=\lm_h\sum\limits_{v_0\in (U\cap V)\w v}v_0f
=\lm_h\cdot |U\cap V|\cdot \w vf=|U\cap V|\cdot v^{-1}f\cdot\w vf~.
\]
On the left-hand side, we obtain
\[\left(\sum\limits_{v_0\in V} U\!hv_0\cdot v_0f\right)\chi_{Uz}
=\left(\sum\limits_{v_0\in V} U\!vv_0\cdot v_0f\right)\chi_{U\w v}
=\sum\limits_{\begin{subarray}{c}v_0\in V\\ Uvv_0=U\w v\end{subarray}}v_0 f
=\sum\limits_{\begin{subarray}{c}v_0\in V\\ vv_0\w v^{-1}\in U\end{subarray}}v_0 f=\sum\limits_{v_0\in V\cap  v^{-1}U\w v}v_0f~.
\] 
We have $V\cap v^{-1}U\w v= v^{-1}\left(v V\w v^{-1}\cap U\right) \w v= v^{-1} (U\cap V) \w v.$

So, $$\sum\limits_{v_0\in V\cap v^{-1}U\w v}v_0f=\sum\limits_{v_0\in v^{-1}(U\cap V)\w v}v_0f=\sum\limits_{c\in U\cap V}(v^{-1}c\w v)f=\sum\limits_{c\in U\cap V}v^{-1}f\cdot \underbrace{cf}_{\begin{subarray}{c}=1\\ \text{ by (J\,1)}\end{subarray}}\cdot\, \w vf=|U\cap V|\cdot v^{-1}f\cdot\w vf\,,$$ which is the same as on the right-hand side.\vsp{1}

{\it Case 2: $h\not\in UV$ and $z\in UV$.}

Write $z=\w u\w v$ with $\w u\in U$ and $\w v\in V$.

On the right-hand side, we obtain, by definition of $\lambda_h$\,, that
$ \left(\lm_h\cdot \sum\limits_{v_0\in V} Uv_0\cdot v_0f\right)\chi_{Uz}
=0~.$

On the left-hand side, we obtain
$\left(\sum\limits_{v_0\in V} U\!hv_0\cdot v_0f\right)\chi_{Uz}
=\left(\sum\limits_{v_0\in V} U\!hv_0\cdot v_0f\right)\chi_{U\w v}
=\sum\limits_{\begin{subarray}{c}v_0\in V\\ Uhv_0=U\w v\end{subarray}}v_0 f~.
$ 

Now, for $v_0\in V$, we have $Uhv_0=U\w v$ if and only if there exists $u\in U$ such that $hv_0=u\w v$. But then $h=u\w v v_0^{-1}\in UV$, which is not the case.

So, $\sum\limits_{\begin{subarray}{c}v_0\in V\\ Uhv_0=U\w v\end{subarray}}v_0 f=0\,,$ which is the same as on the right-hand side.\vsp{1}

{\it Case 3: $h\in UV$ and $z\not\in UV$.}

Write $h=uv$ with $u\in U$ and $v\in V$.

On the right-hand side, we obtain 
$\left(\lm_h\cdot \sum\limits_{v_0\in V} Uv_0\cdot v_0f\right)\chi_{Uz}
=\lm_h\sum\limits_{\begin{subarray}{c}v_0\in V\\ Uv_0=Uz\end{subarray}}v_0 f~.
$

Now, for $v_0\in V$, we have $Uv_0=Uz$ if and only if $z\in Uv_0\subseteq UV$, which is not the case.

So, $\left(\lm_h\cdot \sum_{v_0\in V} Uv_0\cdot v_0f\right)\chi_{Uz}=0~.$

On the left-hand side, we obtain
$ \left(\sum\limits_{v_0\in V} U\!hv_0\cdot v_0f\right)\chi_{Uz}
=\left(\sum\limits_{v_0\in V} U\!vv_0\cdot v_0f\right)\chi_{Uz}
=\sum\limits_{\begin{subarray}{c}v_0\in V\\ Uvv_0=Uz\end{subarray}}v_0 f~.
$

Now, for $v_0\in V$, we have $Uvv_0=Uz$ if and only if $z\in Uvv_0\subseteq UV$, which is not the case.

So, $\left(\sum\limits_{v_0\in V} U\!hv_0\cdot v_0f\right)\chi_{Uz}=0\,,$ which is the same as on the right-hand side.\vsp{1}

{\it Case 4: $h\not\in UV$ and $z\not\in UV$.}

On the right-hand side, we obtain $\left(\lm_h\cdot \sum\limits_{v_0\in V} Uv_0\cdot v_0f\right)\chi_{Uz}=0\,,$ by definition of $\lambda_h$\,.

On the left-hand side, we have $\left(\sum\limits_{v_0\in V}Uhv_0\cdot v_0f\right)\chi_{Uz}=\sum\limits_{\begin{subarray}{c}v_0\in V\\Uhv_0=Uz\end{subarray}}v_0f$.

For $v_0\in V$, we have \[
	Uhv_0=Uz
	\Leftrightarrow hv_0=uz \text{ for some $u\in U$}
	\Leftrightarrow h=uzv_0^{-1} \text{ for some $u\in U$.}
\]
{\it Subcase $h\not\in UzV$.} We obtain $\left(\sum\limits_{v_0\in V}Uhv_0\cdot v_0f\right)\chi_{Uz}=0\,,$ which is the same as on the right-hand side.

{\it Subcase $h\in UzV$.} 
We have \[
	Uhv_0=Uz\Leftrightarrow hv_0\in Uz
	\Leftrightarrow v_0\in h^{-1}Uz~.
\]
Note that $h\in T'\setminus (UV)=T$. By Lemma \ref{LemGI5}, we have $|wN\cap h^{-1}Uz|=|N\cap h^{-1}Uz|$ for $w\in V$ and therefore 
\[\begin{array}{rclcl}\left(\sum\limits_{v_0\in V}Uhv_0\cdot v_0f\right)\chi_{Uz}
&=&\sum\limits_{\begin{subarray}{c}v_0\in V\\Uhv_0=Uz\end{subarray}}v_0f
&=&\sum\limits_{v_0\in V\cap h^{-1}Uz}v_0f\\&=&\sum\limits_{w\in V/N}\sum\limits_{v_0\in wN\cap h^{-1}Uz}v_0f
&\stackrel{\rm D.\ref{DefGI1}.(2)}{=}&\sum\limits_{w\in V/N}|wN\cap h^{-1}Uz|\cdot wf\\
&\stackrel{\rm L.\ref{LemGI5}}{=}&|N\cap h^{-1}Uz|\cdot\sum\limits_{w\in V/N}wf&\stackrel{\rm L.\ref{muK}(2)}{=}&0\,,\end{array}\] which is the same as on the right-hand side. 
\qed

\subsection{James triples}
Since in the Scalar Lemma, Lemma \ref{LemGI7}, we have obtained an assertion for $h\in T'_{(U,V,N)}$, we consider the case $T'_{(U,V,N)}=G$.
\subsubsection{Definition and first properties of James triples}

\label{SecDefJamesTr}
Recall that given a triple of subgroups $\mathcal{T}=(U,V,N)$ in $G$ satisfying (J\,1, 2) we have
\[\begin{array}{rcl}
	T=T_{(U,V,N)}&=&\{h\in G\ohne(UV): \text{we have }|wN\cap U^h| \= |N\cap U^h| \text{ for $w\in V$}\}\subseteq G\ohne(UV)\\
	T'=T'_{(U,V,N)}&=& T_{(U,V,N)}\cup (UV)\subseteq G~,\end{array}\] cf.\ \S\ref{setting}.
\begin{Definition}\label{triples}
\rm
Suppose given $k\in\Z_{\ge 2}\,$.

A {\it James triple} in $G$ of index $k$ is a triple $\mathcal{T}=(U,V,N)$ of subgroups in $G$ such that the conditions (J\,1,\,2,\,3) hold.

\begin{itemize}
\item[(J\,1)] We have $U\cap V\le N$.
\item[(J\,2)] We have $N\nrme V$ such that $|V/N|=k$. 
\item[(J\,3)] We have $T'_{(U,V,N)}=G$. I.e.\ for $h\in G\ohne (UV)$ and for $w\in V$, we have \[
|wN\cap U^h| \= |N\cap U^h|\; ~.
\]
\end{itemize}
\end{Definition}
\begin{Remark}\label{cyc}\rm
Given $U\le N\nrme G$, the triple $(U, G, N)$ is a James triple in $G$ of index $|G/N|$.
\end{Remark}

\begin{Lemma}
\label{RemGI3}
Given a James triple $(U,V,N)$ of index $k$ and $h\in G\ohne (UV)$, we have \[k\cdot |N\cap U^h| = |V\cap U^h|~.\]
\end{Lemma}

{\it Proof.} We have the disjoint union $V = \Disj_{w\in V/N} w N$. 

So, we have the disjoint union $V\cap U^h = \Disj_{w\in V/N} (w N\cap U^h)$. 

By (J\,3), we conclude that $|V\cap U^h| = k\cdot |N\cap U^h|$.
\qed

\begin{Lemma}
\label{RemGI4}
Suppose given a triple $(U,V,N)$ of subgroups such that {\rm (J\,1,\,2)} hold and such that $|V/N| = 2$. 

Then $(U,V,N)$ is a James triple in $G$ of index $2$ if and only if $N\cap U^h \nrme V\cap U^h$ for $h\in G\ohne (UV)$.
\end{Lemma}

{\it Proof.}
If $(U,V,N)$ is a James triple of index $2$ in $G$, then  $N\cap U^h \nrme V\cap U^h$ for $h\in G\ohne (UV)$; cf.\ Lemma~\ref{RemGI3}.

Suppose that $N\cap U^h \nrme V\cap U^h$ for $h\in G\ohne (UV)$. 

We have $\dfrac{|V\cap U^h|}{|N\cap U^h|}\cdot |N|=|(V\cap U^h)\cdot N|\le |V|$. 

So $\dfrac{|V\cap U^h|}{|N\cap U^h|}\in\{1,2\}$. Hence, by assumption, $|V\cap U^h|=2|N\cap U^h|$.

Suppose given $w\in V$.

If $w\in N$, then we have $|wN\cap U^h| \= |N\cap U^h|\; .$

If $w\not\in N$, then we have $|wN\cap U^h| \= |(V\ohne N)\cap U^h| \= |V\cap U^h| - |N\cap U^h| \= |N\cap U^h|\; .$

Hence, (J\,3) holds. So $(U,V,N)$ is a James triple of index $2$ in $G$.
\qed

\subsubsection{Absolutely irreducible modules via James triples}
\label{SecSimpJamesTr}
We recall the permutation $KG$-module $M_K = K(U\< G)$ and that we write \mbox{$m_1 = U\cdot 1 \in M_K\,$.}

Suppose given a James triple $\mathcal{T}=(U,V,N)$ in $G$ of index $k$, cf.\ Definition \ref{triples}. 

Suppose given a group morphism $f : V\to\upmu(K)$  fitting $\mathcal{T}$ with respect to the automorphism $\al_K:K\xrightarrow{\sim} K, x\mapsto\overline{x}$, cf.\ Definition \ref{DefGI1}.

\begin{Definition}\label{Def_S_K}\rm
	We define the $KG$-submodule of $M_K$
	\[
	S^{\mathcal{T},\, f}_K \; :=\;\spi{\sum\limits_{v\in V}m_1v\cdot vf}_{KG}=\;\spi{\sum\limits_{v\in V}Uv\cdot vf}_{KG}\;\le M_K .
	\]

We often abbreviate $S_K=S^{\mathcal{T},\, f}_K$.

Let $S_K^\bot:=\{x\in M_K :[x,S_K]_K=0\}\subseteq M_K$. As $[-,=]$ is $G$-invariant this is a $KG$-submodule. 
\end{Definition}

\begin{Lemma}\label{zero}
	The module $S^{\mathcal{T},\, f}_K$ is non-zero if and only if $\car(K)$ does not divide $|U\cap V|$.
\end{Lemma}
{\it Proof.}
Note that for $w, \tilde w\in V$, we have $Uw=U\tilde w$ if and only if $w\tilde w^{-1}\in U\cap V$, i.e.\ if and only if $(U\cap V)w=(U\cap V)\tilde w$.
We have, using $U\cap V\le N\le\ker(f)$, \[\begin{array}{rcl}\sum\limits_{v\in V}Uv\cdot vf
&=&\hsp{-2}\sum\limits_{w\in (U\cap V)\< V}\sum\limits_{c\in U\cap V}Ucw\cdot (cw)f
=\sum\limits_{w\in (U\cap V)\< V}Uw\left(\sum\limits_{c\in U\cap V}cf\right)\cdot wf\\
&=&|U\cap V|\sum\limits_{w\in (U\cap V)\< V}Uw\cdot wf~.\end{array}\]
As $\sum\limits_{w\in (U\cap V)\< V}Uw\cdot wf\ne 0$ , $S_K$ is zero if and only if $\car(K)$ divides $|U\cap V|$.\qed

\begin{Lemma}\label{quasiid}
Write $U^{\Sigma}=\sum\limits_{u\in U} u\in KG$ and $\xi^{\Sigma}:=\sum\limits_{v\in V} U^{\Sigma}\cdot v\cdot vf\in KG$. 

\begin{enumerate}
\item[\rm(1)] Suppose that $\car(K)$ does not divide $|U|$. 

Then $\spi{\sum\limits_{v\in V}U v\cdot vf}_{KG}=\spi{m_1\xi^{\Sigma}}_{KG}$.
\item[\rm(2)] We have \[\left(\xi^{\Sigma}\right)^2=\left(\sum\limits_{v\in V}\sum\limits_{u\in U}\lm_{vu}\cdot vf\right) \xi^{\Sigma}\] with $\lambda_{vu}\in K$ as in Lemma \ref{LemGI7}. 
\item[\rm(3)] Suppose that $V\le \NN_G(U)$. Then \[\left(\xi^{\Sigma}\right)^2=|V|\cdot |U|\cdot \xi^{\Sigma}~.\] In particular, if  $\car(K)$ divides neither $|U|$ nor $|V|$, the element $\xi^{\Sigma}$ is a quasi-idempotent element in $KG$.
\end{enumerate}
\end{Lemma}
{\it Proof.}
{\it Ad }(1). We have $\spi{\sum\limits_{v\in V}Uv\cdot vf}_{KG}=\spi{m_1\cdot|U|\sum\limits_{v\in V}v\cdot vf}_{KG}=\spi{m_1\sum\limits_{v\in V}U^{\Sigma}\cdot v\cdot vf}_{KG}$.

{\it Ad }(2). By the Scalar Lemma, Lemma \ref{LemGI7} we have \[
\sum\limits_{v\in V}U^{\Sigma}\cdot hv\cdot vf=\lm_h\left(\sum\limits_{v\in V}U^{\Sigma}\cdot v\cdot vf\right)
\] for $h\in G$.

Then, using Lemma \ref{LemGI7} with $h=vu$, we have  \[\begin{array}{rclcl}
\left(\xi^{\Sigma}\right)^2&=& \left(\sum\limits_{v\in V}U^{\Sigma}\cdot v \cdot vf\right)\cdot\left(\sum\limits_{v'\in V} U^{\Sigma}\cdot v' \cdot v'f\right)
     &=&  \sum\limits_{v\in V}\sum\limits_{v'\in V}\sum\limits_{u\in U} U^{\Sigma}\cdot  vu v' \cdot v'f\cdot vf\\[3mm]
     &=&\sum\limits_{v\in V}\sum\limits_{u\in U}\left(\sum\limits_{v'\in V} U^{\Sigma}\cdot  vu v' \cdot v'f\right)\cdot vf
     &\stackrel{\rm L.\ref{LemGI7}}{=}&\sum\limits_{v\in V}\sum\limits_{u\in U}\left(\lm_{vu}\sum\limits_{v'\in V} U^{\Sigma}\cdot  v' \cdot v'f\right)\cdot vf\\[3mm]
     &=&\left(\sum\limits_{v\in V}\sum\limits_{u\in U}\lm_{vu}\cdot vf\right) \sum\limits_{v'\in V} U^{\Sigma}\cdot  v' \cdot v'f
     &=&\left(\sum\limits_{v\in V}\sum\limits_{u\in U}\lm_{vu}\cdot vf\right) \xi^{\Sigma}~.\end{array}\]

{\it Ad }(3). As $V\in \NN_G(U)$ we have $vu=u^{v^{-1}}\cdot v$ with $u^{v^{-1}}\in U$ for $u\in U$, $v\in V$. Therefore,
\[\begin{array}{rclcl}
\left(\xi^{\Sigma}\right)^2&\stackrel{\rm(2)}& \left(\sum\limits_{v\in V}\sum\limits_{u\in U}\lm_{vu}\cdot vf\right) \xi^{\Sigma}
                           &=&\left(\sum\limits_{v\in V}\sum\limits_{u\in U}\lm_{u^{v^{-1}}\cdot v}\cdot vf\right) \xi^{\Sigma}\\[3mm]
                            &\stackrel{\rm L.\ref{LemGI7}}{=}&\left(\sum\limits_{v\in V}\sum\limits_{u\in U}v^{-1}f\cdot vf\right) \xi^{\Sigma}
                            &=&|V|\cdot |U|\cdot \xi^{\Sigma}~.\hfill \qed\end{array}\]
\bq The following proposition is due to {\sc James}, adapted to the present context, cf.\mbox{ \cite[Theorem~4.8]{Jam}.}\eq
\begin{Proposition}\label{simple} Recall the non-degenerate \mbox{$G$-invariant} sesquilinear form $[-,=]_K$ on $M_K$, cf.\ Definition \ref{ses}. Let $X\subseteq M_K$ be a $K G$-submodule. Then 
	\[S^{\mathcal{T},\, f}_K\subseteq X \quad \text{or} \quad X\subseteq {\left(S^{\mathcal{T},\, f}_K\right)}^{\bot}~.\]
\end{Proposition}
{\it Proof.}
	{\it Case 1:}	Suppose that there exists $x\in X$ and $g\in G$ such that\[0\neq \sum\limits_{v\in V}x\cdot v^g\cdot vf\in X~.\]
	
	Write $x=\sum\limits_{h\in G}a_hUh$ with $a_h\in K$. Then \[\begin{array}{rclcl}
	0&\neq& \sum\limits_{v\in V}x\cdot v^g\cdot vf
	 &=&\left(\sum\limits_{v\in V}x\cdot g^{-1} v\cdot vf\right) g\\[3mm]
	 &=&\left(\sum\limits_{h\in G}a_hUhg^{-1}\sum\limits_{v\in V} v\cdot vf\right) g
	 &=&\left(\sum\limits_{\tilde h\in G}a_{\tilde h g}\sum\limits_{v\in V}U\tilde h v\cdot vf\right) g\\[4mm]
	 &\stackrel{\rm L.\ref{LemGI7}}{=}&\left(\sum\limits_{\tilde h\in G}a_{\tilde h g}\lambda_{\tilde h}\sum\limits_{v\in V}U v\cdot vf\right) g~,
	\end{array}\] for certain $\lambda_{\tilde h}\in K$, as in Lemma \ref{LemGI7}.
     
	Hence, there exists $\lm\in K^\ti$ such that $\sum\limits_{v\in V}x\cdot v^g\cdot vf = \lm\sum\limits_{v\in V}{m_1\cdot v\cdot vf} \cdot g$.
	
	 So, in this case $S_K=\langle \sum\limits_{v\in V} m_1\cdot v\cdot vf \rangle_{KG}=\langle \sum\limits_{v\in V} m_1\cdot v\cdot vf \cdot g \rangle _{KG}=\langle \sum\limits_{v\in V}x\cdot v^g\cdot vf \rangle _{KG}\subseteq X$.
	
	{\it Case 2:}	Suppose that \[\sum\limits_{v\in V}x\cdot v^g\cdot vf=0\]for $x\in X$ and $g\in G$.
	
	We show that $X\subseteq S_K^{\bot}$, i.e.\ that $[X,S_K]=0$. Let $x\in X$ and $g\in G$. 
	
	We have \[\begin{array}{rllclcl}
	&&[x,\left(\sum\limits_{v\in V}m_1\cdot v\cdot vf\right)g]&=&[x,\sum\limits_{v\in V}m_1\cdot g\cdot g^{-1}\cdot v\cdot g\cdot vf]\\
	&=&	\sum\limits_{v\in V}[x,m_1\cdot g\cdot g^{-1}\cdot v\cdot g]\overline{vf}
	&=&	\sum\limits_{v\in V}[x\cdot (g^{-1}\cdot v\cdot g)^{-1},m_1\cdot g]\overline{vf}\\
	&=& [\sum\limits_{v\in V}x\cdot (g^{-1}v \cdot g)^{-1}\cdot \overline{vf},m_1\cdot g]
	&\stackrel{\rm D. \ref{DefGI1}.(1)}{=}& [\sum\limits_{v\in V}x\cdot (g^{-1}\cdot v \cdot g)^{-1}\cdot (v^{-1})f,m_1\cdot g]\\
	&=&[\sum\limits_{v\in V}x\cdot g^{-1}\cdot v^{-1}\cdot g\cdot (v^{-1})f,m_1\cdot g]
	&=&[\sum\limits_{v\in V}x\cdot v^g\cdot vf,m_1\cdot g]	\\
	&=&[0,m_1\cdot g]
	&=&0~.\hfill \qed
	\end{array}\]
\begin{Corollary}\label{irred}
	The $KG$-module \[S^{\mathcal{T},\, f}_K/(S^{\mathcal{T},\, f}_K\cap {\left(S^{\mathcal{T},\, f}_K\right)}^\bot)\] is irreducible or zero.
\end{Corollary}

{\it Proof.}
	Every submodule of $S_K/(S_K\cap S_K^\bot)$ is of the form $X/(S_K\cap S_K^\bot)$ for a submodule $X\subseteq M_K$ with $S_K\cap S_K^\bot\subseteq X\subseteq S_K$. We show that $X=S_K\cap S_K^\bot$ or $X=S_K$.
		By Proposition \ref{simple} we have \[S_K\subseteq X \text{ or }X\subseteq S_K^\bot~.\]
	If $S_K\subseteq X$, then $X=S_K$. If $X\subseteq S_K^\bot$ then $X\subseteq S_K\cap S_K^\bot$, hence $X=S_K\cap S_K^\bot$ .
\qed

\begin{Theorem}\label{T1} Recall the James triple $\mathcal{T}=(U,V,N)$ in $G$ of index $k$ and the group morphism $f : V\to\upmu(K)$  fitting $\mathcal{T}$ with respect to the automorphism $\al_K:K\xrightarrow{\sim} K, x\mapsto\overline{x}$.

	The $KG$-module $$S^{\mathcal{T},\, f}_K/(S^{\mathcal{T},\, f}_K\cap {\left(S^{\mathcal{T},\, f}_K\right)}^{\bot})$$ is absolutely irreducible or zero.
\end{Theorem}

{\it Proof.}
	Let $\hat{K}|K$ be an algebraic closure of $K$. By \cite[Chapitre V, \S 9, Proposition 1]{Bou} there exists $\alpha_{\hat{K}}\in\Aut(\hat{K})$ such that $\al_{\hat{K}}|_K^K=\al_K$. Consider the $\hat{K} G$-module $M_{\hat{K}}:=\hat{K}(U\backslash G)=\hat{K}\otimes_K K(U\backslash G)$.
	We have \[S_{\hat{K}}= \spi{\sum\limits_{v\in V}Uv\cdot vf}_{\hat{K}G}\subseteq M_{\hat{K}}~.\]

	We need to show that $\hat{K}\otimes_K (S_K/(S_K\cap S_K^{\bot}))$ is irreducible or zero.
	
	As tensoring with $\hat{K}$ is exact and compatible with intersections, we have a $\hat{K}G$-module isomorphism \[\begin{array}{rcl}\hat{K}\otimes_K \left(S_K/(S_K\cap S_K^{\bot})\right)&\xrightarrow[\sim]{\nu}& S_{\hat{K}}/(S_{\hat{K}}\cap S_{\hat{K}}^{\bot})\\
	\ell\otimes \left(s+(S_K\cap S_K^{\bot})\right) &\mapsto& \ell\otimes s+(S_{\hat{K}}\cap S_{\hat{K}}^{\bot})~.\end{array}\] 
		
	Applying Corollary \ref{irred} yields that the $\hat{K}G$-module $S_{\hat{K}}/(S_{\hat{K}}\cap S_{\hat{K}}^{\bot})$ is irreducible or zero.
	Hence, $\hat{K}\otimes_K (S_K/(S_K\cap S_K^{\bot}))$ is irreducible or zero. 
	
	This shows that the $KG$-module $S_K/(S_K\cap S_K^{\bot})$ is absolutely irreducible or zero. \qed

\begin{Remark}\label{irred4}\rm
Suppose that $K|\mathbb{Q}$ is a subfield of $\mathbb{C}$ that is closed under the complex conjugation $\alpha_{\mathbb{C}}$, i.e.\ $\kappa\alpha_{\mathbb{C}}\in K$ for $\kappa\in K$. Let $\alpha_K=\alpha_\C|^{K}_K$. 

Then \[S^{\mathcal{T},\, f}_K\cap {\left(S_K^{\mathcal{T},\, f}\right)}^\bot=0~,\]
as $[-,=]_K : M_K\times M_K\to K$ is positive definite, cf.\ Remark \ref{possesMK}. 

So $S^{\mathcal{T},\, f}_K$ is absolutely irreducible, cf.\ Lemma \ref{zero}, Theorem \ref{T1}.
\end{Remark}

\section{Examples}

\subsection{\texorpdfstring{James triples in case of $G=\SSS_n$}
{James triples in case of G = Sₙ}}

\bq The construction of the Specht modules of $\SSS_n$ \cite{Jam} can be phrased in terms of James triples. In fact, James triples have been modelled on that case.\eq

Let $n,r\in\Z_{\ge 2}$ and $G=\SSS_n$. Recall that $K$ is a field and $\al_K\in\Aut(K)$.

Let $\mu=(\mu_1,\mu_2,\mu_3,\dots,\mu_r)$ be a partition of $n$ with $r$ parts, written as $\mu\vdash n$.  Let $t$ be a $\mu$-tableau.
Let $U_t\le\SSS_n$ be the row stabiliser of $t$. Let $V_t\le\SSS_n$ be the column stabiliser of $t$. Let $\kappa_t:=\sum\limits_{v\in V_t} v\cdot v\sgn\in K\SSS_n$ be the {\it signed column sum}.

\begin{Definition}\rm
Let $M_K^\mu$ be the $K\SSS_n$-permutation module with basis $\{\{t\}:\, \text{ $t$ is a $\mu$-tabloid}\}$ \cite[4.1]{Jam}.

The Specht module $S^\mu_K$ is the $K\SSS_n$-submodule of $M^\mu_K$ generated by the $\mu$-polytabloid $e_t:=\{t\}\kappa_t$ for a $\mu$-tableau $t$ \cite[4.5]{Jam}. The resulting submodule is independent of the choice of the $\mu$-tableau $t$.
\end{Definition}

We shall now explain the connection of the construction via James triples to James's original construction.

\begin{Remark}\label{isoJ}
Suppose given a $\mu$-tableau $t$. We have the isomorphism of $K\SSS_n$-modules

\[\begin{array}{rcl} M_K=K(U_t\backslash\SSS_n)&\xraiso{\phi_t}& M_K^\mu \\
                         U_t\cdot h        &\mapsto& \{t\}\cdot h~, \text{ where $h\in\SSS_n$}  
\end{array}\]
\end{Remark}

\begin{Lemma}\label{S-trip}\rm
Let $t$ be a $\mu$-tableau. Let $U:=U_t\le\SSS_n$ be the row stabiliser of $t$. Let $V:=V_t\le\SSS_n$ be the column stabiliser of $t$.
Let $f=\signum|_V: V\to \{-1,1\}:\, v\mapsto v\signum$. Let $N:=\ker(f)$, so elements of $N$ are the even permutations within the column stabiliser~$V$.

Then: 

\begin{enumerate}
\item[(1)] The triple $\mathcal{T}=(U,V,N)$ is a James triple of index $2$.
\item[(2)] We have the isomorphism of $K\SSS_n$-modules
          \[\begin{array}{rcl} S_K^{\mathcal{T},\,f}&\to& S^\mu_K \\
                                      \sum\limits_{v\in V} Uv\cdot vf     &\mapsto& e_t~.\end{array}\]
\end{enumerate}
\end{Lemma}
{\it Proof.}
{\it Ad }(1). We have $U\cap V=1$, as only the element $1=\id_{\SSS_n}$ stabilizes the rows and columns of $t$ simultaneously. We have $|V/N|=2$, as $\mu$ has at least two parts, and therefore $t$ has at least two elements in the first column.
So, (J\,1,\,2) hold.

It remains to show that (J\,3) holds. Suppose given $h\in \SSS_n\ohne (UV)$. By Lemma \ref{RemGI4} it suffices to show that $N\cap U^h \nrme V\cap U^h$.

Since $N\nrme V$, we have $N\cap U^h\nrm V\cap U^h$. We need to show that there exist at least one element in $V\cap U^h$ of odd sign. It suffices to find an element of the form $(a,b)$ in $V\cap U^h$.

The group $U^h$ is the row stabiliser of the $\mu$-tableau $th$. 

{\it Assume:} Whenever two elements $a,b\in [1,n]$ are in the same column of $t$, they are in different rows of $th$.

We proceed as in \cite[3.7, 4.6]{Jam}, see also \cite[Lemma 63]{NK}.

Multiplying $h$ with an element $v_1\in V$, we achieve that in $thv_1$, each element of the first column of $t$ is in the same row in $thv_1$ as in $t$.
Multiplying $hv_1$ with an element $v_2\in V$, we achieve that in $thv_1v_2$, each element of the first two colums of $t$ is in the same row in $thv_1v_2$ as in $t$. Continuing this way, we obtain an element $v=v_1\cdot v_2\cdot{\dots}\cdot v_{\lambda_1}$ such that $hv\in U$. Hence $h\in UV$, \textit{contradicting} $h\not\in UV$.

All in all $N\cap U^h \nrme V\cap U^h$ for $h\in \SSS_n\ohne (UV)$. So $\mathcal{T}=(U,V,N)$ is a James triple of index~$2$.

{\it Ad }(2). 
We have $S_K^{\mathcal{T},\,f}=\langle  \sum\limits_{v\in V} Uv\cdot vf \rangle_{K\SSS_n}=\langle  U\kappa_t \rangle_{K\SSS_n}$ and
$(U\kappa_t)\phi_t=\{t\}\kappa_t=e_t$, cf.\ Lemma \ref{isoJ}. So, $\phi_t|_{S_K^{\mathcal{T},\,f}}^{S^\mu_K}:\, S_K^{\mathcal{T},\,f}\to S^\mu_K$ is surjective and thus an isomorphism.
\qed

\begin{Remark}\rm
Suppose that $K=\Q$. The Specht modules over $\Q$ give all the ordinary irreducible representations of $\SSS_n$, cf.\ \cite[Theorem 4.12]{Jam}.

So, by Lemma \ref{S-trip}, every ordinary irreducible representations of $\SSS_n$ can be realised via a James triple.
\end{Remark}

\subsection{Modules not generated by a James triple}
\begin{Example}\label{TabelleTripQuad}\rm
We consider the case $K=\Q$. We give some examples which show that there are nontrivial absolutely irreducible modules that can not be generated by a James triple. We have used the computer algebra system Magma \cite{Magma} and its SmallGroups library to obtain these examples.

Recall that if $G$ is of odd order, then $G$ has exactly one absolutely irreducible representation over $\Q$, namely the trivial module \cite[Ch.\ V, Theorem 13.8 (Burnside)]{Hupp}. 
So we could restrict our search to the case where $G$ is of even order.

The following table lists all groups $G$ of order $n$ at most $190$ for which at least one absolutely irreducible nontrivial $\Q G$-module could not be generated via a James triple. For each such group, every absolutely irreducible $\Q G$-module that could not be generated is also listed.  
The group is specified by its number $k$ in the SmallGroups-library of Magma and the Magma function GroupName. The module is specified by its character given by Magma, where the conjugacy classes are ordered as in \verb|CharacterTable(SmallGroup(n,k))|.

\[
{\scriptsize\hspace*{-12mm}\begin{array}{c|c|c|c|l}
\text{Order of $G$} & \text{SmallGroups} & \begin{array}{l}\text{GroupName(G)}\\\text{as in Magma }\end{array} & \text{Dim} & \text{Character}                    \\\hline
48                  & 17                                       & \text{Q}_8:\SSS_3          & 4 & ( 4, -4, 0, -2, 0, 0, 2, 0, 0, 0, 0, 0 )                                 \\\hline
54					&  6    								   & \DD_9:\CC_3		  & 6 & ( 6, 0, -3, 0, 0, 0, 0, 0, 0, 0 )                                 \\\hline
72                  & 23									   & \CC_3^2:\DD_4        & 4 & ( 4, -4, 0, 0, -2, -2, 1, 0, 2, 2, -1, 0, 0, 0, 0 )\\\hline
96                  & 15                                       & \CC_3:(\CC_8:\CC_4)    & 4 &  ( 4, -4, -4, 4, -2, 0, 0, 0, 0, 0, 0, -2, 2, 2, 0, 0, 0, 0, 0, 0, 0, 0, 0, 0 )                                                \\\hline
96                  & 16                                       &\CC_3:(\DD_4:\CC_4)     & 4 & ( 4, -4, -4, 4, 0, 0, -2, 0, 0, 0, 0, 2, 2, -2, 0, 0, 0, 0, 0, 0, 0, 0, 0, 0
    )\\\hline
96                  & 42                                       &(\CC_3*\text{Q}_8):\CC_4     & 4 &( 4, -4, 4, -4, -2, 0, 0, 0, 0, 0, 0, -2, 2, 2, 0, 0, 0, 0, 0, 0, 0, 0, 0, 0
    )
\\\hline
96                  &148                                       &\CC_2*\text{Q}_8:\SSS_3       & 4 & ( 4, -4, 4, -4, 0, 0, -2, 0, 0, 0, 0, 2, 2, -2, 0, 0, 0, 0, 0, 0, 0, 0, 0, 0
    )\\\hline
96                  &148                                       &\CC_2*\text{Q}_8:\SSS_3       & 4 &( 4, -4, -4, 4, 0, 0, -2, 0, 0, 0, 0, -2, 2, 2, 0, 0, 0, 0, 0, 0, 0, 0, 0, 0
    )
\\\hline
96                  &193                                       &\text{SL}(2,3):\CC_2:\CC_2    & 4 & ( 4, -4, 0, 0, 0, -2, 0, 0, 2, 0, 0, 0, 0 )\\\hline
96                  &201                                       &\text{Q}_8.A_4           & 4 & ( 4, -4, 0, 0, 0, -2, -2, 0, 0, 0, 0, 2, 2, 0, 0, 0, 0, 0, 0 )\\\hline
108                  &9                                       & \CC_{18}.\CC_6          & 6 & ( 6, 6, -3, 0, 0, 0, 0, -3, 0, 0, 0, 0, 0, 0, 0, 0, 0, 0, 0, 0 )\\\hline
108                  &26                                       & \CC_2*\DD_9:\CC_3          & 6 & ( 6, 6, 0, 0, -3, 0, 0, -3, 0, 0, 0, 0, 0, 0, 0, 0, 0, 0, 0, 0 )\\\hline
108                  &26                                       & \CC_2*\DD_9:\CC_3          & 6 &( 6, -6, 0, 0, -3, 0, 0, 3, 0, 0, 0, 0, 0, 0, 0, 0, 0, 0, 0, 0 ) \\\hline
144                &  18                                      &  \text{Q}_8:\DD_9       & 4 &( 4, -4, 0, 4, 0, 0, -4, 0, 0, -2, -2, -2, 0, 0, 0, 2, 2, 2, 0, 0, 0, 0, 0, 
    0, 0, 0, 0 ) \\\hline
144                &  57                                      &  \CC_3^2:\DD_8       & 4 &( 4, 4, 0, 0, -2, -2, 1, -4, -2, -2, 1, 0, 0, 0, 0, 2, 2, -1, -1, 2, 0, 0, 
    0, 0 ) \\\hline
144                &  58                                      &  \DD_{12}.\SSS_3       & 4 &( 4, -4, 0, -2, 4, -2, 0, 0, -4, 2, 2, 0, 0, 0, 0, 0, 0, 0, 0, 0, 0 ) \\\hline
144                &  59                                      &  \CC_4.\SSS_3^2      & 4 &( 4, 4, 0, -2, -2, 1, -4, 0, -2, -2, 1, 0, 0, 0, 0, 2, 2, -1, -1, 2, 0, 0,
    0, 0 )
 \\\hline
 144                &  60                                      &  \CC_3:(\CC_8:\SSS_3)      & 4 &( 4, -4, 0, -2, 4, -2, 0, 0, 2, -4, 2, 0, 0, 0, 0, 0, 0, 0, 0, 0, 0, 0, 0, 0
    )
 \\
                  &                                        &        & 4 &( 4, 4, 0, -2, -2, 1, -4, 0, -2, -2, 1, 0, 0, 2, 2, 2, -1, -1, 0, 0, 0, 0,
     0, 0 )
 \\\hline
  144                &  62                                      &  \CC_3:(\CC_{24}.\CC_2)     & 4 &( 4, 4, -2, -2, 1, -4, 0, 0, -2, -2, 1, 0, 0, 2, 2, 2, -1, -1, 0, 0, 0, 0, 
    0, 0 )
 \\\hline
 144                &  64                                      &  (\CC_6*\SSS_3):\CC_4      & 4 &( 4, -4, -4, 4, 0, 0, -2, -2, 1, 0, 0, 0, 0, -2, -2, 2, 2, 2, 2, -1, -1, 1, 
    0, 0, 0, 0, 0, 0, 0, 0 )
 \\\hline
 144                &  65                                      &  \CC_3:(\DD_6:\CC_4)      & 4 &( 4, -4, -4, 4, 0, 0, -2, -2, 1, 0, 0, 0, 0, 2, 2, -2, -2, 2, 2, -1, 1, -1,
    0, 0, 0, 0, 0, 0, 0, 0 ),\\ & & & &    ( 4, -4, 4, -4, 0, 0, -2, -2, 1, 0, 0, 0, 0, 2, -2, 2, 2, -2, 2, -1, -1, 1, 0, 0, 0, 0, 0, 0, 0, 0 ) \\\hline
 144                &  66                                      &  \CC_6.\DD_{12}     & 4 &( 4, 4, -4, -4, -2, -2, 1, 0, 0, 0, 0, 0, 0, 2, 2, 2, -2, -2, 2, -1, -1, 1, 
    0, 0, 0, 0, 0, 0, 0, 0 )
 \\\hline
 144                &  82                                      &  \CC_3*\text{Q}_8:\SSS_3     & 4 & ( 4, -4, 0, 4, 4, -2, -2, -2, 0, 0, -4, -4, 2, 2, 2, 0, 0, 0, 0, 0, 0, 0, 0,
    0, 0, 0, 0, 0, 0, 0, 0, 0, 0, 0, 0, 0 )
 \\\hline
 144                &  98                                      &  \CC_3:(\text{Q}_8:\SSS_3)     & 4 & ( 4, -4, 0, -2, -2, -2, 4, 0, 0, -4, 2, 2, 2, 0, 0, 0, 0, 0, 0, 0, 0, 0, 0, 0, 0, 0, 0 ),\\
  & & & &  ( 4, -4, 0, 4, -2, -2, -2, 0, 0, 2, -4, 2, 2, 0, 0, 0, 0, 0, 0, 0, 0, 0, 0, 0, 0, 0, 0 ),\\
    & & & &   ( 4, -4, 0, -2, 4, -2, -2, 0, 0, 2, 2, -4, 2, 0, 0, 0, 0, 0, 0, 0, 0, 0, 0, 0, 0, 0, 0 ),\\
     & & & &  ( 4, -4, 0, -2, -2, 4, -2, 0, 0, 2, 2, 2, -4, 0, 0, 0, 0, 0, 0, 0, 0, 0, 0, 0, 0, 0, 0 )
 \\\hline
  144                &  125                                      &  \CC_3:\text{GL}(2,3) &4& ( 4, -4, 0, -2, 1, 1, -2, 0, 2, 2, -1, -1, 0, 0, 0 ),\\
    & & & &    ( 4, -4, 0, -2, -2, 1, 1, 0, 2, -1, 2, -1, 0, 0, 0 ),\\
      & & & &  ( 4, -4, 0, -2, 1, -2, 1, 0, 2, -1, -1, 2, 0, 0, 0 )
 \\\hline
 144                &  151                                      &  \CC_2*\CC_3^2:\DD_4 &4& ( 4, 4, -4, -4, 0, 0, 0, 0, -2, -2, 1, 0, 0, -2, -2, 2, 2, 2, 2, -1, 1, -1, 
    0, 0, 0, 0, 0, 0, 0, 0 ),\\
    & & & &    ( 4, -4, 4, -4, 0, 0, 0, 0, -2, -2, 1, 0, 0, 2, 2, -2, -2, 2, 2, -1, -1, 1, 
    0, 0, 0, 0, 0, 0, 0, 0 ) \\\hline
 160                &  85                                      &  (\CC_5*\text{Q}_8):\CC_4     & 8 & ( 8, -8, 0, 0, 0, 0, 0, 0, 0, 0, 0, 0, -2, 0, 0, 2, 0, 0, 0 )
 \\\hline
 162                &  5                                      &  \CC_3:\DD_9:\CC_3     & 6 & ( 6, 0, -3, -3, 6, -3, 0, 0, 0, 0, 0, 0, 0, 0, 0, 0, 0, 0, 0, 0, 0 )\\
  & & & &  ( 6, 0, -3, 6, -3, -3, 0, 0, 0, 0, 0, 0, 0, 0, 0, 0, 0, 0, 0, 0, 0 )
 \\\hline
  162                &  6                                      &  \DD_9:\CC_9     & 6 & ( 6, 0, 6, 6, -3, -3, -3, 0, 0, 0, 0, 0, 0, 0, 0, 0, 0, 0, 0, 0, 0, 0, 0, 0,
    0, 0, 0, 0, 0, 0 )
 \\\hline
 162                &  36                                      &  \CC_3*\DD_9:\CC_3     & 6 & ( 6, 0, 6, 6, -3, -3, -3, 0, 0, 0, 0, 0, 0, 0, 0, 0, 0, 0, 0, 0, 0, 0, 0, 0,
    0, 0, 0, 0, 0, 0 )
 \\\hline
 162                &  42                                      &  \CC_9:(\CC_3*\SSS_3)     & 6 & ( 6, 0, -3, -3, 6, -3, 0, 0, 0, 0, 0, 0, 0, 0, 0, 0, 0, 0, 0, 0, 0 )\\
  & & & &  ( 6, 0, -3, 6, -3, -3, 0, 0, 0, 0, 0, 0, 0, 0, 0, 0, 0, 0, 0, 0, 0 )\\
  & & & &  ( 6, 0, 6, -3, -3, -3, 0, 0, 0, 0, 0, 0, 0, 0, 0, 0, 0, 0, 0, 0, 0 )
\end{array}}\]

Note that the groups $G$ of order $48$ have, in total, $356$ absolutely irreducible nontrivial $\Q G$-modules;
of order $54$ in total $48$;
of order $64$ in total $3229$;
of order $72$ in total $316$;
of order $96$ in total $2513$;
of order $108$ in total $231$;
of order $128$ in total $42256$;
of order $144$ in total $1973$;
of order $160$ in total $1774$;
of order $162$ in total $240$.
\end{Example}
\begin{Example}\label{TabelleTripQuad3}\rm
We consider the case of a finite group $G$ of order $n$, of a prime divisor $p$ of $n$ and of  $K=\Fu{p^k}$ for $k\in\{1,2,3,4\}$.
We have used the computer algebra system Magma \cite{Magma} and its SmallGroups library to obtain these examples.

We may restrict our search to the case where $G$ is not a $p$-group, cf.\ \mbox{\cite[Ch.\ V, Theorem 5.16]{Hupp}.}

The following table lists all group orders $\le 190$ for which at least one absolutely irreducible nontrivial $\Fu{p^k} G$-module could not be generated via a James triple. We tested for $k=2$ up to order $n=143$, for $k=3$ up to order $n=93$ and $k=4$ up to order $n=36$.

\[{\scriptsize\begin{array}{c|c|c|c||}
k&\shortstack{Group\\order} &
\shortstack{Number of nontrivial\ru{-0.7}\\absolutely irreducible\\modules in total} &
\shortstack{Number of those\ru{-0.7}\\ modules not generated\\by a James triple}
\\ \hline
\smatev{1}{2}{3}{4} & 10 & \smatev{2}{4}{2}{8}    & \smatev{0}{2}{0}{0} \\ \hline
\smatev{1}{2}{3}{4} & 12  & \smatev{11}{21}{11}{21} & \smatev{1}{0}{1}{0} \\ \hline
\smatev{1}{2}{3}{4} & 14 & \smatev{2}{2}{11}{2} & \smatev{0}{0}{9}{0} \\ \hline
\smatev{1}{2}{3}{4} & 18 & \smatev{11}{25}{14}{25} & \smatev{0}{0}{3}{0} \\ \hline
\smatev{1}{2}{3}{4} & 20 & \smatev{16}{20}{16}{28} & \smatev{6}{10}{6}{6} \\ \hline
\smatev{1}{2}{3}{4} & 21 & \smatev{4}{6}{4}{6} & \smatev{4}{6}{4}{6} \\ \hline
\smatev{1}{2}{3}{4} & 24 & \smatev{61}{99}{61}{99} & \smatev{6}{11}{6}{11} \\ \hline
\smatev{1}{2}{3}{4} & 28 & \smatev{8}{12}{26}{12} & \smatev{0}{0}{18}{0} \\ \hline
\smatev{1}{2}{3}{4} & 30 & \smatev{12}{36}{12}{76} & \smatev{0}{12}{0}{20} \\ \hline
\smatev{1}{2}{3}{4} & 34 & \smatev{2}{2}{2}{10} & \smatev{0}{0}{0}{8} \\ \hline
\smatev{1}{2}{3}{4} & 36 & \smatev{43}{93}{49}{93} & \smatev{6}{0}{12}{0} \\ \hline
\smated{1}{2}{3} & 39 & \smated{8}{8}{20} & \smated{8}{8}{20} \\ \hline
\smated{1}{2}{3} & 40 & \smated{73}{99}{73} & \smated{25}{51}{25} \\ \hline
\smated{1}{2}{3} & 42 & \smated{36}{48}{66} & \smated{18}{22}{48} \\ \hline
\smated{1}{2}{3} & 48 & \smated{320}{534}{320} & \smated{35}{75}{35} \\ \hline
\smated{1}{2}{3} & 50 & \smated{5}{21}{5} & \smated{0}{16}{0} \\ \hline
\smated{1}{2}{3} & 52 & \smated{15}{15}{18} & \smated{6}{6}{9} \\ \hline
\smated{1}{2}{3} & 54 & \smated{48}{148}{63} & \smated{1}{0}{16} \\ \hline
\smated{1}{2}{3} & 55 & \smated{10}{10}{10} & \smated{10}{10}{10} \\ \hline
\smated{1}{2}{3} & 56 & \smated{46}{70}{103} & \smated{2}{2}{59} \\ \hline
\smated{1}{2}{3} & 57 & \smated{4}{4}{4} & \smated{4}{4}{4} \\ \hline
\smated{1}{2}{3} & 60 & \smated{84}{196}{84} & \smated{20}{75}{20} \\ \hline
\smated{1}{2}{3} & 63 & \smated{20}{24}{32} & \smated{20}{24}{32} \\ \hline
\smated{1}{2}{3} & 68 & \smated{15}{19}{15} & \smated{6}{10}{6} \\ \hline
\smated{1}{2}{3} & 70 & \smated{8}{16}{32} & \smated{0}{8}{24} \\ \hline
\smated{1}{2}{3} & 72 & \smated{251}{443}{275} & \smated{35}{37}{59} \\ \hline
\smated{1}{2}{3} & 77 & \smated{0}{0}{6} & \smated{0}{0}{6} \\ \hline
\smated{1}{2}{3} & 78 & \smated{42}{52}{102} & \smated{26}{28}{86} \\ \hline
\smated{1}{2}{3} & 80 & \smated{447}{575}{447} & \smated{179}{307}{179} \\ \hline
\smated{1}{2}{3} & 84 & \smated{137}{215}{227} & \smated{64}{96}{154} \\ \hline
\smated{1}{2}{3} & 88 & \smated{45}{69}{45} & \smated{2}{10}{2} \\ \hline
\smated{1}{2}{3} & 90 & \smated{44}{168}{56} & \smated{0}{44}{12}\\ \hline
\smated{1}{2}{3} & 93 & \smated{4}{4}{4}    & \smated{4}{4}{4} \\ \hline
\smatez{1}{2} & 96  & \smatez{2197}{3671} & \smatez{228}{586} \\ \hline
\smatez{1}{2} & 100 & \smatez{59 }{105}  & \smatez{20}{66} \\ \hline
\smatez{1}{2} & 104 & \smatez{71}{83}   & \smatez{24}{36}
\end{array}\begin{array}{c|c|c|c}
k &
\shortstack{Group\\order} &
\shortstack{Number of nontrivial \ru{-0.7}\\absolutely irreducible\\modules in total} &
\shortstack{Number of those\ru{-0.7}\\ modules not generated\\by a James triple}
\\ \hline
\smatez{1}{2} & 105 & \smatez{4 }{12}   & \smatez{4}{8} \\\hline
\smatez{1}{2} & 108 & \smatez{183}{515}  & \smatez{35}{2} \\ \hline
\smatez{1}{2} & 110 & \smatez{54 }{60}  & \smatez{40}{46} \\ \hline
\smatez{1}{2} & 111 & \smatez{4}{4}    & \smatez{4}{4} \\ \hline
\smatez{1}{2} & 112 & \smatez{253}{423}  & \smatez{21}{31} \\ \hline
\smatez{1}{2} & 114 & \smatez{32}{40}   & \smatez{16}{16} \\ \hline
\smatez{1}{2} & 116 & \smatez{15}{15}   & \smatez{6}{6} \\ \hline
\smatez{1}{2} & 117 & \smatez{28}{28}   & \smatez{28}{28} \\ \hline
\smatez{1}{2} & 120 & \smatez{528}{1156}  & \smatez{125}{504} \\ \hline
\smatez{1}{2} & 126 & \smatez{178}{260}  & \smatez{113}{137} \\ \hline
\smatez{1}{2} & 129 & \smatez{4}{4}    & \smatez{4}{4} \\ \hline
\smatez{1}{2} & 130 & \smatez{8}{28}    & \smatez{0}{20} \\ \hline
\smatez{1}{2} & 132 & \smatez{55}{115}   & \smatez{4}{2} \\ \hline
\smatez{1}{2} & 136 & \smatez{92}{100}   & \smatez{44}{52} \\ \hline
\smatez{1}{2} & 140 & \smatez{58}{100}   & \smatez{12}{42} \\ \hline
1 & 144 & 1444 & 212 \\ \hline
1 & 147 & 12   & 12 \\ \hline
1 & 148 & 15   & 6 \\ \hline
1 & 152 & 45   & 2 \\ \hline
1 & 155 & 8    & 8 \\ \hline
1 & 156 & 222  & 144 \\ \hline
1 & 160 & 3230 & 1331 \\ \hline
1 & 162 & 240  & 7 \\ \hline
1 & 164 & 15   & 6 \\ \hline
1 & 165 & 12   & 12 \\ \hline
1 & 168 & 801  & 330 \\ \hline
1 & 171 & 40   & 40 \\\hline
1 & 176 & 246  & 17 \\\hline
1 & 180 & 344  & 84 \\\hline
1 & 182 & 14   & 6 \\\hline
1 & 183 & 4    & 4 \\\hline
1 & 184 & 45   & 2 \\\hline
1 & 186 & 32   & 16 \\\hline
1 & 189 & 140  & 140 \\\hline
1 & 190 & 12   & 4 \\\hline
\phantom{1}& &
\end{array}}\]

So for instance, for all groups $G$ of order $n=20$ and $p\in\{2,5\}$ there are in total $16$ nontrivial absolutely irreducible $\Fu{p} G$-modules up to isomorphism, of which $6$ are not generated by a James triple;
there are $20$ nontrivial absolutely irreducible $\Fu{p^2} G$-modules, of which $10$ are not generated by a James triple;
there are $16$ nontrivial absolutely irreducible $\Fu{p^3} G$-modules, of which $6$ are not generated by a James triple;
there are $28$ nontrivial absolutely irreducible $\Fu{p^4} G$-modules, of which $6$ are not generated by a James triple.

\end{Example}
\bibliographystyle{abbrvurl}
\bibliography{references}

@book{Bou,
  author    = {Nicolas Bourbaki},
  title     = {Éléments de mathématique, Algèbre, Chapitres 4 à 7},
  year      = {2007},
  publisher = {Springer}
}

@book{Hupp,
  author    = {Bertram Huppert},
  title     = {Finite Groups I},
  publisher = {Springer},
  year      = {2025}
}

@book{Jam,
  author    = {Gordon James},
  title     = {The Representation Theory of the Symmetric Groups},
  publisher = {Springer},
  year      = {1978}
}

@book{Jesp,
  author    = {Eric Jespers and Gabriela Olteanu and Angel del Rio},
  title     = {Rational Group Algebras of Finite Groups},
  publisher = {Springer},
  year      = {2010}
}

@phdthesis{NK,
  author = {Nora Krau{\ss}},
  title  = {Some constructions of irreducible modules over finite groups},
  school = {Universit{\"a}t Stuttgart},
  year   = {2026},
  url   = {https://elib.uni-stuttgart.de/handle/11682/19179}
}

@article{Magma,
  author  = {Wieb Bosma and John Cannon and Catherine Playoust},
  title = 	 {The {Magma} algebra system. {I}. {The} user language},
  journal = {Journal of Symbolic Computation},
  volume  = {24},
  year    = {1997},
  pages   = {235--265}
}

@inproceedings{Parker1984,
  author    = {Parker, Richard A.},
  title     = {The computer calculation of modular characters (the {Meat-Axe})},
  booktitle = {Computational Group Theory},
  pages     = {267--274},
  publisher = {Academic Press},
  year      = {1984}
}

@incollection{Parker1998,
  author    = {Parker, Richard A.},
  title     = {An integral ``{Meat-axe}''},
  booktitle = {The Atlas of Finite Groups: Ten Years On},
  series    = {London Mathematical Society Lecture Note Series},
  volume    = {249},
  pages     = {215--228},
  publisher = {Cambridge University Press},
  year      = {1998}}

@article{HoltRees1994,
  author  = {Holt, Derek F. and Rees, Sarah},
  title   = {Testing modules for irreducibility},
  journal = {Journal of the Australian Mathematical Society. Series A},
  volume  = {57},
  number  = {1},
  pages   = {1--16},
  year    = {1994}
}

@article{Shoda1933,
  author    = {Kenjiro Shoda},
  title     = {{{\"U}ber die monomialen Darstellungen einer endlichen Gruppe}},
  journal   = {Proceedings of the Physico-Mathematical Society of Japan. 3rd Series},
  volume    = {15},
  number    = {7-8},
  pages     = {249--257},
  year      = {1933}
}

@phdthesis{Steel2012,
  author     = {Allan Kenneth Steel},
  title      = {Construction of Ordinary Irreducible Representations of Finite Groups},
  school     = {University of Sydney},
  year       = {2012},
  url        = {https://magma.maths.usyd.edu.au/users/allan/reps/AllanSteelPhD.pdf}
}
\vspace*{5mm}

\begin{flushleft}
Nora Krau\ss \\
Universit\"at Stuttgart \\
Fachbereich Mathematik \\
Pfaffenwaldring 57 \\ 
70569 Stuttgart \\
\verb|kraussna@mathematik.uni-stuttgart.de|
\end{flushleft}
\end{document}